\documentclass[twoside]{article}

\usepackage[utf8]{inputenc}
\usepackage[english]{babel}
\usepackage{amsthm}
\usepackage{amsmath}

\theoremstyle{plain}
\newtheorem{theorem}{Theorem}

\theoremstyle{definition}
\newtheorem{definition}{Definition}

\begin{document}

\title{On a matching arrangement of a graph and LP-orientations of a matching polyhedron}
\author{A.I.Bolotnikov
}
\maketitle

\section{Introduction} 

The number of regions of a hyperplane arrangement can be calculated with its characteristic polynomial \cite{zasl}. This result is useful in solving problems of enumerative combinatorics. For example, in \cite{irm}, \cite{irm2}, \cite{irm3}  a special hyperplane arrangement with several related geometric constructions was introduced in order to obtain an asymptotic estimation of  the number of threshold functions and the number of singular $\pm 1$-matrices .

In the paper \cite{gre-zasl} the connection between the chromatic polynomial of a graph and the number of acyclic orientations of that graph is depicted as a property of a specific hyperplane arrangement. This arrangement is called a graphical arrangement, and the set of all edges of the initial graph is used for its construction. The characteristic polynomial of the graphical arrangement is equal to the chromatic polynomial of the initial graph, and the number of regions of the graphical arrangement is equal to the number of acyclic orientations in the initial graph.

In the paper \cite{edmonds} a linear program for a maximum matching problem for a general graph was created. A matching polyhedron was defined by the set of its inequalities, and the correspondence between its vertices and matchings of the initial graph was shown. The concept of LP-orientations is connected to linear programming as well \cite{halt-klee}.

A matching arrangement was defined in the paper \cite{bol}. That paper also describes some of matching arrangement's properties and properties of its characteristic polynomial.

This paper contains a description of a connection between the matching arrangement and the matching polyhedron. A bijection between regions of the matching arragement and LP-orientations of the matching polyhedron is constructed. This bijection allows to calculate the number of LP-orientations of the matching polyhedron with the characteristic polynomial of the matching arrangement.

\section{Basic definitions}

\begin{definition} \textit{A hyperplane arrangement} is a finite set of affine hyperplanes in some
vector space V. In this paper $V = R^n$.
\end{definition}
\begin{definition}If all hyperplanes intersect in one point, an arrangement is called \textit{central} .
\end{definition}
\begin{definition}
\textit{A region} of an arrangement  is a connected component of
the complement  of the union of hyperplanes.
\end{definition}

\begin{definition}

Let $(r_1,r_2,r_3,...,r_n)$ be a sequence of edges, such that $\forall i$ edges $r_i$ and $r_{i+1}$ are incident to the same vertex. If $r_n$ and $r_1$ are incident to the same vertex, then such sequence is called\textit{a cycle}, otherwise it is called  \textit{a path}. A path or a cycle is called simple, if there are no pairs of edges that are incident to the same vertex, other than ($r_i$,$r_{i+1}$) and ($r_1$,$r_n$). \textit{A length} of a path or a cycle is the amount of edges in it.
\end{definition}
\begin{definition}
Let $G(V,E), |E|=n$ be a graph without loops and parallel edges, let $N:E\rightarrow \{1,2,..,n\}$ be a numeration of edges in G. Let P be a sequence of edges $(r_1,r_2,r_3,...,r_n)$ that form a simple path or a simple cycle with even number of edges in graph G. A hyperplane that corresponds to sequence P is a hyperplane with an equation $x_1-x_2+x_3 -...+(-1)^{n+1}x_n=0$. Let F be a set of all sequences of edges in G that form a simple path or a simple cycle with even number of edges in G. Then matching arrangement is a set of all hyperplanes that correspond to elements of F. A matching arrangement will be denoted as MA(G,N).           
\end{definition}
\begin{definition} \cite{edmonds} Let $G(V,E), |E|=n$ be a graph without loops and parallel edges. A matching polyhedron is a convex polyhedron in $R^n$ that is defined by a system of inequalities that consists of inequalities of the following three types:
\begin{enumerate}
    \item $\forall e_i \in E : x_i \geq 0$
    \item $\forall v \in V : \sum\limits_{e_i \text{ is incident } v} x_i \leq 1$
    \item $\forall S=\{v_{j_1},v_{j_2},..., v_{j_{2k+1}} \}, S\subseteq V : \sum\limits_{e_i=(v_p,v_q) \in E,v_p \in S, v_q \in S} x_i\leq k$
\end{enumerate}
\end{definition}
Each matching of a graph $G(V,E), |E| = n$ corresponds to a vertex of an n-dimensional boolean cube in the following way. Let $M$ be a matching in $G$. Then the corresponding vertex of a boolean cube is $(x_1, x_2, x_3, …, x_n)$,where $x_i = 1$, if an edge $e_i$ belongs to M, and $x_i = 0$ otherwise.
Therefore, a bijection between the set of all matchings in $G$ and a subset $D$ of the set of vertices of a boolean cube is constructed.

\begin{theorem}\cite{edmonds} $D$ is a set of vertices of a matching polyhedron.
\end{theorem}

Let $P$ be a polyhedron in $R^n$, and let $\phi: R^n 
\rightarrow R$ be a linear function that takes different values on adjacent vertices of $P$. Then the graph of $P$ can be oriented by $\phi$ the following way. An edge $(v_i, v_j)$ that connects vertices $v_i$ and $v_j$, is oriented from $v_i$ to $v_j$, if $\phi(v_i) < \phi(v_j)$.
If an orientation can be obtained from a linear function this way, then this orientation is called an LP-orientation . 

\section{Main result}

\begin{theorem}
Let $G(V,E) , |E| = n$ be a graph without loops and parallel edges, and let $P(G)$ be a matching polyhedron of $G$. Let $L(P(G))$ be a set of LP-orientations of $P(G)$.
The mapping $F:L(P(G)) \rightarrow 2^{R^{n}}$ is defined the following way. For every LP-orientation O \[F(O)= \{(\alpha_1, \alpha_2, ..., \alpha_n) \in R^n |f(x_1,x_2,…,x_n) =\]\[= \alpha_{1}x_1 + \alpha_{2}x_2 + ... + \alpha_{n}x_n \text{induces O}\}\].

Let $RG(MA(G))$ be a set of regions of the matching arrangement of G, $RG(MA(G)) = \{R_1, R_2,...,R_n\}$.

Then:
\begin{enumerate}
    \item $\forall O \in L(P(G))$  $ F(O)$ is a region of MA(G),that is $F(O)\in RG(MA(G))$. Therefore, F can be considered as a mapping from $L(P(G))$ to $RG(MA(G))$;
    \item $F: L(P(G))\rightarrow RG(MA(G))$ is a bijection.
\end{enumerate}
\end{theorem}

\begin{proof}
Let $\Pi(v)$ be a matching that corresponds to a vertex $v$ of the matching polyhedron. Matching polyhedra have the following property \cite[Theorem 25.3]{schrijver}: vertices $v_1$ and $v_2$ of a matching polyhedron $P(G)$ are adjacent, if and only if $\Pi(v_1)\triangle\Pi(V_2)$ is either a simple path or a simple cycle with an even length.

Let  O be an LP-orientation of $P(G)$, and let $a = (a_1,...,a_n)$ and $b=(b_1,...,b_n)$ be elements of $F(O)$. 
Let's assume,that $a$ and $b$ belong to different regions of MA(G). Then there is either a simple path or simple cycle with an even length $(e_{i_1}, e_{i_2},....,e_{i_k}),e_{i_j} \in E$ ,such that points $a$ and $b$ are separated by the hyperplane $x_{i_1} - x_{i_2} + x_{i_3}...+(-1)^{k+1}x_{i_k} = 0$. Without loss of generality, let $a_{i_1} - a_{i_2} + a_{i_3}...+(-1)^{k+1}a_{i_k} > 0$, $b_{i_1} - b_{i_2} + b_{i_3}...+(-1)^{k+1}b_{i_k} < 0$. By construction, set of edges $\{e_{i_1},e_{i_3},e_{i_5},...\}$ and $\{e_{i_2}, e_{i_4}, e_{i_6},...\}$ are matchings, and the corresponding vertices in $P(G)$ are connected with an edge $t$ of $P(G)$. The edge $t$ is oriented differently in LP-orientations that are obtained from functions  $f_a(x)=a_1 x_1 + a_2 x_2 +....+ a_n x_n$ and $f_b(x)=b_1 x_1 + b_2 x_2 +....+ b_n x_n$, therefore $a$ and $b$ cannot belong to $F(O)$ at the same time. 

Let's assume now, that $a$ belongs to the border of a region R, which means it belongs to a hyperplane $x_{i_1} - x_{i_2} + x_{i_3}...+(-1)^{k+1}x_{i_k} = 0$.Then the function $f_a$ has the same value on vertices of P(G) that correspond to matchings $\{e_{i_1},e_{i_3},e_{i_5},...\}$ and $\{e_{i_2}, e_{i_4}, e_{i_6},...\}$, and are connected by an edge in $P(G)$. This means that an LP-orientation cannot be obtained from $f_a$, and $a$ cannot belong to F(O).

As a result, if O is an element of $L(P(G))$, then $F(O)$ is a subset of some region $R_i$. Let $a \in F(O), b\in R_i$, $b$ does not belong to F(O). Since $b$ does not belong to hyperplanes of the arrangement MA(G), values of $f_b$ on adjacent vertices of P(G) are different, which means that an LP-orientation $\hat O$ can be obtained from $f_b$. Let's assume, that an edge $t$ is oriented differently in O and in $\hat O$. Let $v_1$ and $v_2$ be vertices of t, $\Pi(v_1) = \{e_{j_1},e_{j_3},...\}$, $\Pi(v_2) = \{e_{j_2},e_{j_4},...\}$. Then $H:x_{j_1}-x_{j_2} + x_{j_3}-... = 0$ is a hyperplane from MA(G), and $a$ and $b$ are separated by H, which is a contradiction to an assumption, that $a$ and $b$ belong to the same region $R_i$. Therefore $O = \hat O$,  $F(O)=R_i$, and $F: L(P(G))\rightarrow RG(MA(G))$ is injective.

Let R be a region of MA(G), $b \in R$. Since $b$ does not belong to any hyperplane from MA(G), values of $f_b$ on adjacent vertices of P(G) are different, which means that an LP-orientaion $\bar O$ can be obtained from $f_b$. Then, as it was previously proved, $F(\bar O) = R$. This means that F is surjective. Therefore, F is a bijection.
\end{proof}

Author is grateful to  A.A.Irmatov for the formulation of the problem and for valuable discussions of these results.

\end{document}